\documentclass[journal, 12pt]{IEEEtran}
\onecolumn
\ifCLASSINFOpdf
\else
\fi
\usepackage{amsmath}
\usepackage{amssymb}
\usepackage{amsthm}
\usepackage[colorlinks=true, 
		link color=blue, 
		citecolor=magenta,
		pdfborder={0 0 0}]{hyperref}
\usepackage{cleveref}
\usepackage{tikz,tikz-3dplot}
\definecolor{light-gray}{gray}{0.85}

\begin{document}
\title{Iterative Observer for Boundary Estimation\\ for Elliptic Equations}

\author{Muhammad~Usman~Majeed%,~\IEEEmembership{Member,~IEEE,}
        %John~Doe,~\IEEEmembership{Fellow,~OSA,}
        ~and~Taous~Meriem~Laleg-Kirati %,~\IEEEmembership{Senior Member,~IEEE}% <-this % stops a space
%\thanks{M. Shell was with the Department
%of Electrical and Computer Engineering, Georgia Institute of Technology, Atlanta,
%GA, 30332 USA e-mail: (see http://www.michaelshell.org/contact.html).}% 
\thanks{M. U. Majeed and T. M. Laleg-Kirati are with Computer Electrical and Mathematical Sciences and Engineering (CEMSE) Division at King Abdullah University of Science and Technology (KAUST), K.S.A.}
%<-this % stops a space
}

\maketitle

\begin{abstract}
In this paper we propose the design of an iterative observer using space as a time-like variable and prove its convergence. The iterative observer algorithm solves boundary estimation problem for a steady-state elliptic equation system namely Cauchy problem for Laplace equation.
%In this paper we propose the design and prove the convergence of an iterative observer algorithm using space as a time-like variable to solve boundary estimation problem for a steady-state elliptic system namely Cauchy problem for Laplace equation. 
The Laplace equation is formulated as a first order state space-like system in one of the space variables and an iterative observer is developed that sweeps over the whole domain to recover the unknown data on the boundary. State operator matrix is proved to generate strongly continuous semigroup under certain conditions and the system is shown to be observable. Convergence results of proposed algorithm are established using semigroup theory and concepts of observability for distributed parameter systems. 
The algorithm is implemented using finite difference discretization schemes and numerical implementation is detailed. Further, the simulation results are presented towards the end to show efficiency of the algorithm.
\end{abstract}

%%% Note that keywords are not normally used for peerreview papers.
%%\begin{IEEEkeywords}
%%IEEE, IEEEtran, journal, \LaTeX, paper, template.
%%\end{IEEEkeywords}

% \fi
\IEEEpeerreviewmaketitle

\section{Introduction}
\IEEEPARstart{T}{he} problem to estimate some unknown states of a physical system from some measured data using state observers is well-known in dynamical systems' theory. State observer is an algorithm that provides estimates of internal states of a given real system from measurements of inputs and outputs \cite{c12-1}. 
%It is a well-known concept in dynamical systems' theory. 
Early state observer designs were proposed for state estimation for lumped parameter systems governed by ordinary differential equations (ODEs). However the concepts of observer design have been extended to distributed parameter systems (DPSs) modeled by time varying partial differential equation (PDEs) \cite{new_1, new_2, feb_8}. Traditionally for DPSs early or late lumping techniques are considered \cite{feb_1}. Early lumping techniques transform DPS to a finite dimensional system of ODEs using some approximation and discretization techniques \cite{feb_2, feb_3}. The resultant system of ODEs is an approximation to the DPS and unknown states recovered by the state observers may not be the estimate of true states \cite{feb_4}. On the other hand late lumping techniques exploit mathematical properties of underlying PDEs to develop observer design. Various design techniques based on semigroup theory, spectral theory, Lyapunov based design, backstepping approaches are available \cite{new_2, feb_8, new_3, feb_5, feb_6, feb_7, feb_10, feb_9}. All of these methods and techniques are focussed on time varying systems modeled by hyperbolic or parabolic PDEs. However there has been very little effort to develop observer-like algorithms for systems governed by steady-state elliptic PDEs. One example is \cite{new_4}, where  an extra time variable is introduced to solve steady-state heat conduction problem modeled by elliptic PDE as a parabolic problem. The apparent reason for not tackling steady-state elliptic PDE problems using dynamical systems' inspired methods is the unavailability of time dynamics.

In this paper, the objective is to develop an observer-like iterative algorithm using space as time-like. For this purpose steady-state Laplace equation is represented as an infinite-dimensional linear state-space-like system and boundary state estimation strategy is developed. The goal is to extend the dynamical theory concept of state-observer to steady state boundary value elliptic problems without introducing a particular notion of time and to explore the potential challenges to develop such an algorithm. As per knowledge of the authors such a strategy for time-independent systems governed by PDEs has not been studied in literature previously. The successful implementation of such an algorithm will provide a major step towards the possibility of tackling both steady-state and time varying PDE problems using dynamical systems' techniques in a more uniform manner.

The boundary estimation problem for steady-state elliptic equation, namely Cauchy problem for Laplace equation has been a fundamental problem of interest in many diverse areas of science and engineering. For example non-destructive testing applications in mechanics, where we are interested in finding inside cracks from boundary measurements \cite{b1}. Biomedical applications in finding the actual heart potential from electrocardiogram (ECG) data collected on the body torso. Finding the actual heart potential is vital to understand the functionality of heart valves \cite{c7,c6}. 
%Further there are some geophysical applications \cite{b4}. 
Readers may refer to a number of existing numerical solution techniques for elliptic Cauchy problems for further understanding of nature of the mathematical problem, e.g. \cite{b5,b6,b7,c3,c4,c_4,a2,b10}. Almost all of these techniques can be categorized as optimization based methods whereas the algorithm presented in this paper is based on observer design.%, a concept from dynamical systems' theory \cite{new_1, c12-1}.

The paper is organized as follows. Some notations and definitions are provided in section \ref{section_notation}. Problem formulation and transformation to a control familiar state-space representation is provided in section \ref{formulation}. Iterative observer design and proof of convergence is provided in section \ref{obsv_design}. Numerical implementation using finite difference methods and numerical simulation results are discussed in section \ref{section_numerical}. Finally the paper is concluded with a discussion in section \ref{conclusions}.

\section{Notations and definitions}\label{section_notation}
In this section, let $X$ be a Hilbert space with inner product $\left\langle.,.\right\rangle$ and corresponding norm $\|.\|$. If $X$ and $Y$ are two Hilbert spaces then $\mathcal{L}(X,Y)$ denotes the space of linear operators from $X$ to $Y$ with induced norm. Further $\mathcal{L}(X)=\mathcal{L}(X,X)$. Let an infinite dimensional linear dynamical system be presented in state space representation as,
\begin{equation} 
\label{state-space}
\dot{\xi}(x)=\mathcal{A}\xi(x);\quad y(x)=\mathcal{C}\xi(x);
\end{equation}
such that `` $\dot{}$ '' represents partial derivative with respect to time-like variable $x$, $\xi$ be a state vector, $\mathcal{A}:D(\mathcal{A})\rightarrow X$ be the state operator matrix, $\mathcal{C}\in \mathcal{L}(X,Y)$ be the observation operator with observation space $Y$.

\theoremstyle{definition}
\newtheorem{definition}{Definition}
\begin{definition} \cite{feb_9}\label{definition1}
\quad A family $\mathbb{T}=(\mathbb{T}_x)_{x\geq 0}$ of operators in $\mathcal{L}(X)$ defines a strongly continuous semigroup ($C_0$-Semigroup) on $X$ if,
\begin{enumerate}
\item $\mathbb{T}_0 = I$,\quad\quad\quad\quad\quad\quad\quad\quad\quad\quad\quad \;(identity)
\item $\mathbb{T}_{x+w} = \mathbb{T}_x\mathbb{T}_w,\quad \forall x,w\geq 0$, \quad \;(semigroup property)
\item $\lim_{x\rightarrow 0^{+}} \| \mathbb{T}_{x}\xi -\xi \|=0 \quad \forall \xi\in X$. (strong continuity)
\end{enumerate}
\end{definition}

%%%%%	REMOVED DEFINITIONS   %%%%%

\begin{definition}
Let $\mathcal{C}\in \mathcal{L}(X,Y)$ be the observation operator. For all $\bar{x}>0$, let $\Psi_{\bar{x}}\in\mathcal{L}(X,L^2\left( [0,\bar{x}];Y\right))$ be the output map operator for the system (\ref{state-space}) such that,
\begin{equation}
\left(\Psi_{\bar{x}}\xi(0)\right)(x) = \begin{cases}
\mathcal{C}\mathbb{T}_{x}\xi(0) & \forall \;x\in[0,\bar{x}],\\
0 & \forall \;x>\bar{x}.
\end{cases}
\end{equation}
\end{definition}

\begin{definition}
%\cite{weiss}
\label{definition2} Let $\mathbb{T}$ be the strongly continuous semigroup on space $X$ with generator $\mathcal{A}:D(\mathcal{A})\rightarrow X$ and $\mathcal{C}\in\mathcal{L}(X,Y)$ be the observation operator. The pair $\mathcal{(C,A)}$ is exactly observable in $\bar{x}$ if $\Psi_{\bar{x}}$ is bounded from below.
% Pair $(\mathcal{A,C})$ is final state observable in some horizontal distance (time) $\bar{x}$ if there exists a $k_{\bar{x}}>0$ such that $\| \Psi_{\bar{x}}\xi(0) \| \geq k_{\bar{x}} \| \mathbb{T}_{\bar{x}}\xi(0) \|$.
\end{definition}
\noindent 
%Using the density of D($\mathcal{A}^{\infty}$) in $X$ the 
above definition of exact observability of the pair $(\mathcal{C,A})$ is equivalent to the fact that there exists $k_{\bar{x}}>0$ such that,
\begin{equation}\label{exact_obsv}
\int_0^{\bar{x}}\left\| \Psi_{\bar{x}}\xi(0) \right\|^2 dx \geq k_{\bar{x}}^2 \left\| \xi(0) \right\|^2 \quad\forall \xi(0) \in X.
\end{equation}

\begin{definition}
%\cite{weiss}
\label{definition3}
Pair $(\mathcal{C,A})$ as defined above is final state observable in $\bar{x}$ if there exists a constant $k_{\bar{x}}>0$ such that,
\begin{equation}\label{final_obsv_1}
\| \Psi_{\bar{x}} \xi(0) \| \geq k_{\bar{x}} \| \mathbb{T}_{\bar{x}} \xi(0) \| \quad\forall \xi(0) \in X.
\end{equation}
\end{definition}
%\theoremstyle{note}
%\newtheorem{note}{Note}
%\begin{remark}\label{remark1}
\subsubsection*{Note 1} For $\bar{x}\rightarrow 0$ and given that $k_0>0$, then using strong continuity of operator semgigroup $\mathbb{T}$ we can see that definitions in equation (\ref{exact_obsv}) and (\ref{final_obsv_1}) converge.

\subsection*{Lumer-Phillips Theorem:}
Let $\mathcal{A}: D(\mathcal{A})\rightarrow X$ be an unbounded operator on a Hilbert space X. Then the following two assertions are equivalent.
\begin{enumerate}
	\item $\mathcal{A}$ is maximally dissipative.
	\item $\mathcal{A}$ is the generator of a contraction semigroup $(\mathbb{T}_x)_{x\geq 0}$, i.e. $\| \mathbb{T}_x \|\leq 1$ for all $x>0$.
\end{enumerate}
%\end{note}
%\begin{definition}
%\cite{weiss} \label{approx_obsv}
%Pair $(\mathcal{C,A})$ as defined above is approximately observable in $\bar{x}$ if Ker$\Psi_{\bar{x}}=\{0\}$.
%\end{definition}
%\\~\\
%\noindent Further another remark from \cite{weiss} suggests that exact observability implies the other two concepts of observability.

%%%%%	REMOVED DEFINITIONS   %%%%%

\begin{definition}
\label{concat1}
Let $x\in[c,d)$ for all $c,d\in \mathbb{R}$ and $d>c$ then $x_m$, for all $m\in \tilde{\mathbb{Z}}=\{0\}\cup\mathbb{Z}^+$, represents $x$ over $m^{th}$ iteration over the interval $[c,d)$ .

%Further let $\underset{d-c}{\diamondsuit}$ represents concatenation or joining of two intervals of length $d-c$ and $I_i = [c,d)$ for all $i\in \mathbb{N}$ then
%\begin{equation} \label{concatenation}
%	x_1 \underset{d-c}{\diamondsuit} y \underset{d-c}{\diamondsuit} \hdots \underset{d-c}{\diamondsuit} x_m = x, \quad for \; x\in\cup_{i=1}^{m}I_i.
%\end{equation}
\end{definition}
%\noindent Above concatenation, over the interval of length $d-c$, suggests various iterations over the interval $[c,d)$.
%In general, let us define $x_m$ as $x\in[c,d]$ for $m-$th iteration over the interval $[c,d]$ for $m\geq 1$.
%\bigskip\\

Further without loss of generality, let $s\in[0,\pi/4]$, $\mathcal{A}:\mathit{D}(\mathcal{A})\rightarrow X$ be an unbounded differential operator matrix given as,
\begin{equation}
	\mathcal{A} = \left(\begin{array}{cc} 0 & 1\\ -\dfrac{\partial^2}{\partial s^2} & 0\end{array}\right),
\end{equation}
 such that,
%\begin{eqnarray}
%\label{operator_A}\mathcal{A} &=& 
%\left( \begin{array}{cc} 
%0 & 1\\ -\dfrac{\partial^2}{\partial y^2} & 0
%\end{array} \right),\\
\begin{eqnarray}
\label{space-x} X &=& H_{\Gamma_{T}}^1\left(0,\frac{\pi}{4}\right)\times L^2\left(0,\frac{\pi}{4}\right),\\
\label{space-domain} \mathit{D}(\mathcal{A}) &=& \left[f \in H^2\left(0,\frac{\pi}{4}\right)\cap H_{\Gamma_{T}}^1\left(0,\frac{\pi}{4}\right) \left. \right\rvert \frac{df}{ds}(0)=c_2\right] \nonumber\\ & & \times H_{\Gamma_{T}}^1\left(0,\frac{\pi}{4}\right),
\end{eqnarray}
where,
\begin{equation}
H^1_{\Gamma_T}\left( 0,\frac{\pi}{4}\right) =  \left\lbrace f \in H^1\left( 0,\frac{\pi}{4} \right) \left. \right\rvert f(0)=c_1 \right\rbrace,
\end{equation}
and $c_1,c_2$ are constants (coming from Cauchy data at a particular point on $\Gamma_T$) and $X$ is a Hilbert space with scalar product given by,
\begin{eqnarray}
\label{norm}
\left\langle \left(\begin{array}{c}
q_1\\q_2
\end{array} \right), \left( \begin{array}{c}
p_1\\p_2
\end{array}\right) \right\rangle = \int_0^{\frac{\pi}{4}}\frac{dq_1}{ds}(s)\frac{d\bar{p}_1}{ds}(s)ds \nonumber \\  +\int_0^{\frac{\pi}{4}}q_1(s)\bar{p}_1(s) ds+\int_0^{\frac{\pi}{4}}q_2(s)\bar{p}_2(s)ds.
\end{eqnarray}
It can be seen that $D(\mathcal{A}^{\infty})$ is dense in $X$.

\section{Problem formulation}\label{formulation}
Let $\Omega$ be a rectangular domain in  $\mathbb{R}^2$ with boundaries $\Gamma_T, \Gamma_B, \Gamma_L$ and $\Gamma_R$ as shown in Figure~\ref{domain} such that $\bar{\Omega}=\Omega\cup\Gamma_{T}\cup\Gamma_{B}\cup\Gamma_{L}\cup\Gamma_{R}$, $\Omega=(0,a)\times (0,b)$ and $\Gamma_{T}\cap\Gamma_{B}\cap\Gamma_{L}\cap\Gamma_{R}=\emptyset$. Cauchy problem for Laplace equation is defined as,\\~\\
\indent Find $u(x)$ on $\Gamma_{B}$:
\begin{equation}
\label{inverse1}
\begin{cases}
\triangle u =\dfrac{\partial^2 u}{\partial x^2}+\dfrac{\partial^2 u}{\partial y^2}= 0 & in\; \Omega,\\
u = f(x) & on\; \Gamma_{T},\\
\dfrac{\partial u}{\partial n}=g(x) & on\; \Gamma_{T},
%u = 0 & on \; \Gamma_{L\cup R},
\end{cases}
\end{equation}
\noindent %with homogeneous Dirichlet, Neumann or Robin type of side boundaries
with homogeneous Dirichlet or Neumann side boundaries, $f$ and $g$ are given sufficiently smooth and $\frac{\partial }{\partial n}$ represents the normal derivative to the top boundary $\Gamma_{T}$. %The freedom to choose any of the homogeneous Dirichlet, Neumann or Robin side boundary conditions will become obvious once observer based algorithm is fully elaborated in coming sections. 
For consistent Cauchy data on $\Gamma_B$ problem (\ref{inverse1}) can be solved analytically \cite{b_a1}. However, the objective here is to explore the possibility of developing an observer-like iterative algorithm using space as time-like.

\vspace{1em}
\begin{figure}[thpb]
			\centering
%	\framebox{\parbox{4in}
%	{
	\begin{center}
		\begin{tikzpicture}%[framed]
			\begin{scope}
				\draw[fill=light-gray] (2,0) rectangle (6cm,2.5cm);				
				%\draw (-1,-2) rectangle (5,2.5);
				\put(70,30){$\Omega=(0,a)\times (0,b)$};
				\put(100,-12){$\Gamma_{B}$};
				\put(100,77){$\Gamma_{T}$};
				\put(42,30){${\Gamma_{L}}$};
				\put(35,-12){$(0,0)$};
				\put(35,75){$(0,b)$};
				\put(165,-12){$(a,0)$};
				\put(165,75){$(a,b)$};
				\put(173,30){${\Gamma_{R}}$};
			\end{scope}
		\end{tikzpicture}
	\end{center}
%	}
	\caption{Rectangular domain $\Omega$.}
	\label{domain}
\end{figure}
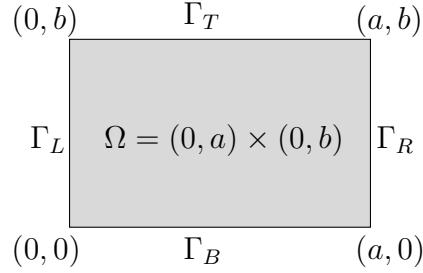

\subsection*{Change of variables:}

%\subsection*{A state-space-like problem formulation}

%\subsection{Abstract formulation}
We propose to write down the Laplace equation in rectangular coordinates as given in system (\ref{inverse1}) as a first order state equation by introducing two new auxiliary variables $\xi_1,\xi_2$ as follows,
\begin{equation}
\label{cases2}
\begin{cases} 
\xi_1(x,y)=u(x,y),\\
\xi_2(x,y)=\dfrac{\partial u}{\partial x },\\
\end{cases}
\end{equation}
and the resulting equation can now be written as,
\begin{equation}
\dfrac{\partial\xi}{\partial x} = \mathcal{A}\xi,
\end{equation}
where, 
\begin{equation}
\label{operator-a} \xi = \left( \begin{array}{c}
\xi_1(x,y) \\ \xi_2(x,y)
\end{array} \right), \quad \mathcal{A} = \left(\begin{array}{cc} 0 & 1 \\ -\dfrac{\partial^2}{\partial y^2} & 0 
\end{array}\right).
\end{equation}
$\xi_1$ and $\xi_2$ are called state variables and using these new variables, problem (\ref{inverse1}) can be written in equivalent form as,\\~\\
\indent Find $\xi_1(x,y)$ on $\Gamma_{B}$:
\begin{equation}
\label{inverse2}
\begin{cases}
\dfrac{\partial \xi}{\partial x} = \mathcal{A}\xi & in \;\Omega,\\
\mathcal{C}\xi(x)=\xi_1(x) = f(x) & on \;\Gamma_{T},\\
\dfrac{\partial \xi_1}{\partial y} = g(x) & on \;\Gamma_{T},%,\\
%\xi_{1} = 0 \; and/or \; \dfrac{\partial \xi_1}{\partial x}=0 & on \; \Gamma_{L\cup R}.
\end{cases}
\end{equation}
with homogeneous Dirichlet/Neumann side boundaries. 

\section{Observer Design} \label{obsv_design}
Boundary value problem as given in system of equations (\ref{inverse2}) has a first order state equation in variable $x$ and overdetermined data is available on $\Gamma_{T}$. Before the introduction of iterative observer equations,
let us assume that left hand boundary $\Gamma_{L}$ is connected to right hand boundary $\Gamma_{R}$ to have the notion of infinite time-like variable $x$ over the rectangular domain. The reason for having such an assumption is that we are trying to develop an observer using space as time-like and hoping that this observer will converge asymptotically in variable $x$. Let $m$ be a non-negative integer index of iteration over the domain $\Omega$ in horizontal direction. Let $x_m$, as given in Definition \ref{concat1}, represents $x\in[0,a)$ for the $m$-th iteration over the interval $[0,a)$.
%and different values of index of iteration $m$ correspond to different intervals of variable $x$ as follows,
%\begin{eqnarray}
%\label{iteration_index}
%m &=& 1\quad\quad x\in[0,a),\\
%m &=& 2\quad\quad x\in[a, 2a)\nonumber,\\
%m &=& 3\quad\quad x\in[2a, 3a)\nonumber,\\
%&\,\;\vdots & \;\quad\quad\quad \vdots\nonumber
%\end{eqnarray}
After introducing iteration index $m$, now an observer-like algorithm can be developed as follows,
\subsubsection*{Main result}
\newtheorem{th1}{Theorem}
\begin{th1}\label{theorem_1}
%For iteration index $m\geq 1$, 
For consistent Cauchy data, boundary value problem given in (\ref{observer1}) asymptotically ($m=1,\cdots, \infty$) converges to the true solution of boundary value problem (\ref{inverse2}).
%The forward space marching algorithm can be written as follows,
\begin{equation}
%\label{space_iterative}
\label{observer1}
\begin{cases}
\vspace{0.5em}
\dfrac{\partial}{\partial x}\hat{\xi}(x_m,y) = \mathcal{A}\hat{\xi}(x_m,y) - \mathcal{KC}(\hat{\xi}(x_m,y) - \xi) & in \;\Omega,\\
%\hat{\xi}_1 = \hat{f}(\theta) & on \;\Gamma_{out},\\
\vspace{0.5em}
\dfrac{\partial}{\partial y}\hat{\xi}_1(x_m,y) = g(x) & on \;\Gamma_{T},\\
\left(\dfrac{\partial^2}{\partial x^2}+\dfrac{\partial^2}{\partial y^2}\right)\hat{\xi}_1(x_m,y) = - \mathcal{KC}(\hat{\xi}(x_m,y)-\xi) & on \;\Gamma_{B},\\
%\hat{\xi}^{m=0} = \xi_0 & in \;\bar{\Omega}.
\hat{\xi}(x_m,y)\mid_{initial} = \hat{\xi}(x_{m-1},y) & in \;\bar{\Omega},
\end{cases}
\end{equation}
where `` $\hat{}$ '' represents estimated quantity and $\hat{\xi}(x_m,y)\mid_{initial}$ represents a bounded estimate over the whole domain $\bar{\Omega}$ at the start of $m$-th iteration. Algorithm starts at index $m=1$, which represents first iteration. $\hat{\xi}(x_0,y)$ is initial guess at the start of the first iteration over the whole domain $\bar{\Omega}$. Any bounded initial guess $\hat{\xi}(x_0,y)$ can be chosen. For each subsequent iteration, result of the previous iteration is used as initial estimate as given in the last equation in (\ref{observer1}).
%To be more obvious, different values of index of iteration $m$ correspond to different values of variable $\theta$ as follows,\\~\\
%\begin{eqnarray}
%\label{iteration_index}
%m &=& 1\quad\quad \theta\in[0,2\pi),\\
%m &=& 2\quad\quad \theta\in[2\pi, 4\pi)\nonumber,\\
%m &=& 3\quad\quad \theta\in[4\pi, 6\pi)\nonumber,\\
%&\,\;\vdots & \;\quad\quad\quad \vdots\nonumber
%\end{eqnarray}
Third equation in (\ref{observer1}) is the assumption that Laplace equation is valid on the bottom boundary and this provides necessary boundary condition required on $\Gamma_{B}$. $\mathcal{C}$ is the observation operator such that $\mathcal{C}\xi = \xi_1\mid_{\Gamma_{T}}$. $\mathcal{K}$ is the correction operator chosen in such a way that state estimation error on $\Gamma_{T}$ given by $(\mathcal{C}\hat{\xi}(x_m,y)-\mathcal{C}\xi)$ converges to zero asymptotically ($m=1,\cdots,\infty$). 
\end{th1}
%State estimation error and the design of correction operator $\mathcal{K}$ is explained in the next section.

\subsection{Preliminary analysis}

Before moving to the proof of theorem (\ref{theorem_1}) we note that the solution of first order equation in system (\ref{inverse2}) leads to the concept of semigroup generated by unbounded differential operator matrix $\mathcal{A}$. We study the exponential of $\mathcal{A}$ using the functional analysis framework from section \ref{section_notation}. 
%Further the basic definitions of observability for infinite-dimensional system are re-visited.
%%Further using the concept of observability for infinite dimensional systems, we establish that pair $\mathcal{(C, A)}$ is final state and exact observable.

\subsubsection{Semigroup generated by $\mathcal{A}$}

\begin{th1}
\label{theorem1}
Let $n\in\mathbb{Z}^{\star}$ (set of non zero integers), for $\mathcal{A}:\mathit{D}(\mathcal{A})\rightarrow X$ (as given in (\ref{operator-a}), (\ref{space-x}) and (\ref{space-domain})) there exists an infinite set of orthonormal eigenvectors ($\Phi_n$) and corresponding eigenvalues ($\lambda_n$). Furthermore $\mathcal{A}$ generates a strongly continuous semigroup for vectors 
$\left(\begin{array}{c}
p_1 \\ p_2
\end{array}\right)\in X$ if and only if decay rate of $\left\langle \left( \begin{array}{c} p_1\\p_2 \end{array} \right), \Phi_n \right\rangle$ is greater than the growth rate of $e^{\lambda_nx}$.
%such that,
%
%\begin{equation}
%	\left(\begin{array}{c}
%		P_1\\P_2
%	\end{array}\right) = \sum_{k=1}^{n} \left( \begin{array}{c}(p_1)_k\\(p_2)_k \end{array}\right),\quad (p_1)_k, (p_2)_k \in C^{\infty};n<\infty.
%\end{equation}

\end{th1}
\begin{proof}
Let, %$n\in\mathbb{Z}$ set of all integers such that,
\begin{equation}
\label{eigenvector}
\Phi_n(y) = \rho_n\left( \begin{array}{c}
\alpha_n\varphi_n(y) \\ \beta_n\varphi_n(y)
\end{array} \right),
\end{equation}
be the orthonormal set of eigenvectors of operator $\mathcal{A}$ and $\lambda_n$ be the eigenvalues such that,

\begin{eqnarray}
\label{eigen1}
\mathcal{A}\Phi_n &=& \lambda_n \Phi_n,\\ \nonumber
%\mathcal{A}\left( \begin{array}{c}
%\alpha_n\varphi_n \\ \beta_n\varphi_n
%\end{array} \right) &= \lambda_n \left( \begin{array}{c}
%\alpha_n\varphi_n \\ \beta_n\varphi_n
%\end{array} \right),\\ \nonumber
\left( \begin{array}{c}
\beta_n\varphi_n \\ -\dfrac{\partial^2}{\partial y^2} \left(\alpha_n\varphi_n\right)
\end{array} \right) &=& \lambda_n \left( \begin{array}{c}
\alpha_n\varphi_n \\ \beta_n\varphi_n
\end{array} \right).
\end{eqnarray}
Assuming that $\alpha_n,\beta_n$ do not depend on $y$, second equation above suggests that we are interested in finding the eigenfunctions of Laplacian operator $\dfrac{-\partial ^2}{\partial y^2}$. This signifies that unknown eigenfunctions $\varphi_n\in C^{\infty}$. Solving two equations in (\ref{eigen1}) gives,
\begin{eqnarray}
\label{eigen2}
\lambda_n &=& \dfrac{\beta_n}{\alpha_n},\\
\label{varphi}\varphi_n(y) &=& C_1\cos\left(\lambda_n y \right),
\end{eqnarray}
where $\alpha_n$ and $\beta_n$ depend on $n$. $C_1$ and $\lambda_n$ are chosen such that $\varphi_n(y)$ in (\ref{varphi}) forms an orthonormal basis in $L^2\left( 0,\frac{\pi}{4} \right)$, with $C_1=-\sqrt{\frac{8}{\pi}}, \alpha_n=1$ and $\beta_n=\lambda_n=6-8n$.
%\begin{equation}
%\label{phi}
%\varphi_n(y) = -\sqrt{\frac{8}{\pi}} \cos\left(\lambda_n y \right),
%\end{equation}
%is an orthonormal basis in $L^2\left( 0,\frac{\pi}{4} \right)$. 
Finally an orthonormal set of eigenvectors can be formed in $X$ with respect to norm defined by (\ref{norm}) as,
\begin{equation}
\label{basis_X}
\Phi_n(y) = \rho_n \phi_n(y) = \rho_n \left( \begin{array}{c}
\alpha_n\varphi_n(y) \\ \beta_n\varphi_n(y)
\end{array} \right),
\end{equation}
where, $|\rho_n|=\dfrac{1}{\left|\sqrt{2}\beta_n\right|}>0$, is a normalization factor.
%Finally (\ref{eigenvector}) are orthonormal eigenvectors in $X$ with norm defined by (\ref{norm}), $\varphi_n$ given by (\ref{phi}) and,
%\begin{eqnarray}
%\alpha_n	&=& 1,\\
%\beta_n 	&=&	6-8n,
%\end{eqnarray}
%hence,
%\begin{equation}
%%\boxed{
%\lambda_n = 6-8n;\quad  \forall n\in\mathbb{Z}
%%}
%\;.
%\end{equation}
% are infinite set of eigenvalues and $\phi_n$ are corresponding orthonormal eigenvectors in $X$ with normalization factor $\rho = 1/ \left| (6-8n) \sqrt{1+\pi/8}\right|$. 
Now let us try to write semigroup generated by operator matrix $\mathcal{A}$ can be written as an infinite series,
\begin{equation} \label{series}
\sum_{n\in\mathbb{Z}^{\star}} e^{\lambda_n x} \left\langle \left(\begin{array}{c}
p_1(y)\\p_2(y)
\end{array}\right),\Phi_n(y)\right\rangle\Phi_n(y), \quad\quad\forall\;\left(\begin{array}{c}
p_1\\p_2
\end{array}\right)\in X.
\end{equation}
For $x=0$ the above infinite series is clearly convergent, whereas for $x\rightarrow 0^+$ the limit does not exist. Further we note that above series expression (\ref{series}) satisfies identity and semigroup properties as given in Definition \ref{definition1}, however it lacks strong continuity, except if we assume that the projection terms in angle brackets above decay faster than the growth rate of $e^{\lambda_n x}$. This condition true for a wide range of analytical functions that have a finite number of non-zero projections on the basis $\Phi_n$.
%are non-zero for some large $n\in\mathbb{Z}:n<\infty$. Now with the introduction of this assumption the limit $x\rightarrow 0^+$ exists. 
This also reveals a historical fact about solving Cauchy problems for steady state heat equation that unique and stable solutions does not exist for non-smooth data \cite{b_a1}. Thus with this additional smoothness assumption 
%that $\left(\begin{array}{c} p_1\\p_2 \end{array}\right) \in X$ 
equation (\ref{semigroup1}) represents the strongly continuous semigroup generated by operator matrix $\mathcal{A}$.
\begin{eqnarray} \label{semigroup1}
\mathbb{T}_{x} \left(\begin{array}{c}
p_1(y)\\p_2(y)
\end{array}\right)&=&\sum_{n\in\mathbb{Z}^{\star}} e^{\lambda_n x} \left\langle \left(\begin{array}{c}
p_1(y)\\p_2(y)
\end{array}\right),\Phi_n(y)\right\rangle\Phi_n(y), \nonumber \\ & &\forall\;\left(\begin{array}{c}
p_1\\p_2
\end{array}\right)\in X.
\end{eqnarray}
This implies,
\begin{equation}
\begin{split}
\label{semigroup}
\mathbb{T}_{x}\left(\begin{array}{c}
p_1(y)\\p_2(y)
\end{array}\right)&=\sum_{n\in\mathbb{Z}^{\star}} e^{\lambda_n x} \rho_n \left(\alpha_n\left\langle\dfrac{dp_1}{dy},\frac{d\varphi_n}{dy}\right\rangle_{\mathrm{L}^2\left(0,\frac{\pi}{4}\right)}\right.
\\
& \left. +\alpha_n\left\langle p_1,\varphi_n\right\rangle_{\mathrm{L}^2\left(0,\frac{\pi}{4}\right)}
+\beta_n\left\langle p_2,\varphi_n\right\rangle_{\mathrm{L}^2\left(0,\frac{\pi}{4}\right)}\right)\Phi_n,\nonumber \\ &\forall\;\left(\begin{array}{c}
p_1\\p_2
\end{array}\right)\in X.
\end{split}
\end{equation}
Q.E.D.
\end{proof}

\subsubsection{System observability}
%Let $\bar{x}>0$ and $\Psi_{\bar{x}}\in\mathcal{L}(X,L^2\left( [0,\bar{x}];Y\right))$ be the output map operator with $Y=\mathbb{R}$ such that,
%\begin{equation}
%\left(\Psi_{\bar{x}}\xi(0)\right)(x) = \begin{cases}
%\mathcal{C}\mathbb{T}_{\bar{x}}\xi(0) & for \;x\in[0,\bar{x}],\\
%0 & for \;x>\bar{x}
%\end{cases}
%\end{equation}
%
%%\theoremstyle{definition}
%%\newtheorem{definition}{Definition}
%\begin{definition}
%Pair $\mathcal{(C,A)}$ is exectly observable in some horizontal distance (time) $\bar{x}$ if $\Psi_{\bar{x}}$ is bounded from below.
%% Pair $(\mathcal{A,C})$ is final state observable in some horizontal distance (time) $\bar{x}$ if there exists a $k_{\bar{x}}>0$ such that $\| \Psi_{\bar{x}}\xi(0) \| \geq k_{\bar{x}} \| \mathbb{T}_{\bar{x}}\xi(0) \|$.
%\end{definition}
%This definition of exact observability of pair $(\mathcal{C,A})$ is equivalent to the fact that there exists $k_{\bar{x}}>0$ such that,
%\begin{equation}
%\int_0^{\bar{x}}\left\| C\mathbb{T}_{x}\xi(0) \right\|^2 dx \geq k_{\bar{x}}^2 \left\| \xi(0) \right\|^2
%\end{equation}

%\theoremstyle{proposition}
\newtheorem{proposition}{Proposition}

\begin{proposition}\label{propos_1}
Let $\mathbb{T}$ be the strongly continuous semigroup generated by operator matrix $\mathcal{A}$ under the assumptions as given in theorem \ref{theorem1}. For any arbitrarily small $\epsilon>0$ such that if $\left| \bar{x}-x \right|<\epsilon$, the pair $(\mathcal{C},\mathcal{A})$ is final state observable (and further exactly observable using Note 1 from section \ref{section_notation}) in time-like interval $\left|\bar{x}-x \right|>0$ at a particular $x$, where $\mathcal{C}\in\mathcal{L}(X,Y)$ and $Y=\mathbb{R}$.
\end{proposition}

\begin{proof}
Let $\xi(0)$ be the initial guess at $x=0$, given by,
\begin{equation}
\xi(0) = \left(\begin{array}{c}
\xi_1(0)\\ \xi_2(0)
\end{array}\right)=\left(\begin{array}{c}
p_1(y)\\p_2(y)
\end{array}\right).
\end{equation}
$\Phi_{n}(y)$ for $n\in\mathbb{Z}^{\star}$ be an orthonormal basis in $X$. Let us first prove the final state observability condition for a general mode $\Phi_{n'}$ with corresponding eigenvalue $\lambda_{n'}$ as follows,\newline
For all $\Phi_{n'}\in X$ and $n'\in\mathbb{Z}^{\star}$,
\begin{eqnarray}\label{obsv_(1)}
\left\| \mathbb{T}_{x} \Phi_{n} \right\|_{X}&=&\left\|\sum_{n\in\mathbb{Z}^{\star}}e^{\lambda_{n} x}\left\langle \Phi_{n'},\Phi_{n} \right\rangle\Phi_{n}\right\|_{X}, \nonumber \\ %\forall\;\Phi_{n'}\in X,\;\;n'\in\mathbb{Z}^{\star},\nonumber\\
&=&e^{\lambda_{n} x} \left\| \Phi_{n} \right\|_{X},\nonumber\\
&=&e^{\lambda_{n} x},
\end{eqnarray}
also,
\begin{eqnarray}\label{obsv_(2)}
\left\| \mathcal{C}\mathbb{T}_{x} \Phi_{n} \right\|_{{Y}} &=& \left\|\sum_{n\in\mathbb{Z}^{\star}}e^{\lambda_{n} x}\left\langle \Phi_{n'},\Phi_{n} \right\rangle\mathcal{C}\Phi_{n}\right\|_{{Y}}, \nonumber \\ %\forall\;\Phi_{n'}\in X,\;\;n'\in\mathbb{Z}^{\star},\nonumber\\
 &=& e^{\lambda_{n} x} \left\| \mathcal{C}\Phi_{n} \right\|_{{Y}},\nonumber\\
% &=& e^{\lambda_{n} x} | \rho_{n} |,\nonumber\\
 &=& e^{\lambda_{n} x} | \rho_{n} |.
\end{eqnarray}
Comparing equations (\ref{obsv_(1)}) and (\ref{obsv_(2)}) implies,
\begin{equation} \label{final_obsv}
\left\| \mathcal{C}\mathbb{T}_{x} \Phi_{n} \right\|_{Y} \geq k_1^{\star} \left\| \mathbb{T}_{x} \Phi_{n} \right\|_{X},
\end{equation}
where $k_1^{\star}>0$, if and only if,
\begin{equation}
k_1^{\star} \leq |\rho_{n}|,
\end{equation}
for a particular choice of $\Phi_{n}$ there always exists $k_1^{\star}$ such that final state observability condition (\ref{final_obsv_1}) is satisfied.\\
$\mathcal{C}\in\mathcal{L}(X,Y)$ is a linear boundary observation operator. Now let $\xi(0)=\sum_{n\in\mathbb{Z}^{\star}}\gamma_{n}\Phi_{n}$ where $\gamma_{n}$ are projection terms whose decay rate is greater than the growth rate of $e^{\lambda_{n}x}$ with $\lambda_{n}$ as eigenvalues of $\mathcal{A}$ corresponding to eigenvectors $\Phi_{n}$. Clearly $\sum_{n\in \mathbb{Z}^{\star}}\gamma_{n}$ and $\sum_{n\in\mathbb{Z}^{\star}}\rho_{n}$ are bounded from above, hence,
\begin{equation} \label{final_obsv_2}
\left\| \mathcal{C}\mathbb{T}_{x} \xi(0) \right\|_{Y} \geq k_2^{\star} \left\| \mathbb{T}_{x} \xi(0) \right\|_{X},
\end{equation}
where $k_1^{\star},k_2^{\star}$ both are independent of $x$. Further using note 1 for arbitrarily small $\epsilon$ pair ($\mathcal{C},\mathcal{A}$) is exactly observable.
Q.E.D.
\end{proof}

\subsection{Convergence analysis}
After establishing the concept of strongly continuous semigroup generated by $\mathcal{A}$ and the fact that pair $\mathcal{(C, A)}$ is final state and exact observable, we are all set to prove the main result.
\subsubsection*{Proof of the main result}
\begin{proof}
Let us define state estimation error $\tilde{e}(x_m,y)$ as the difference of  true state $\xi(x,y)$ from the one estimated $\hat{\xi}(x_m,y)$,
\begin{equation}
\label{true1}
\tilde{e} = \hat{\xi}-\xi = \left( \begin{array}{c}
\tilde{e}_1(x_m,y) \\ \tilde{e}_2(x_m,y)
\end{array} \right) = \left( \begin{array}{c}
\hat{\xi}_1(x_m,y) - \xi_1(x,y)\\ \hat{\xi}_2(x_m,y) - \xi_2(x,y)
\end{array} \right).
\end{equation}
%Let us suppose following nice mixed boundary value problem for Laplace equation,
%\begin{equation}
%\label{analytical_sol}
%\begin{cases}
%\vspace{0.5em}
%\dfrac{\partial \xi}{\partial x} = \mathcal{A}\xi & in \;\Omega,\\
%%\xi_1 = f(\theta) & on \;\Gamma_{out},\\
%\dfrac{\partial \xi_1}{\partial y} = g(x) & on \;\Gamma_{T},\\
%\xi_1 = h(x) & on \;\Gamma_{B},\\
%\end{cases}
%\end{equation}
%with homogeneous Dirichlet and/or Neumann boundary conditions on $\Gamma_{L\cup R}$. 
Solution of the boundary value problem (\ref{inverse2}) with consistent boundary data provides $u=\xi_1$ over the whole domain $\bar{\Omega}$. Boundary value problem for the state estimation error can be given by subtracting problem (\ref{inverse2}) from the state observer equations (\ref{observer1}) as follows, \\~\\
For $m\geq 1$, find $\tilde{e}(x_m,y)=(\hat{\xi}(x_m,y)-\xi(x,y))\in \bar{\Omega}$:
\begin{equation}
\begin{cases}
\vspace*{0.6em}\dfrac{\partial}{\partial x} \tilde{e}(x_m,y) = (\mathcal{A-KC})\tilde{e}(x_m,y) & in \;\Omega,\\
%e_1 = 0 & on \;\Gamma_{out},\\
\vspace*{0.6em}\dfrac{\partial}{\partial y}\tilde{e}_1(x_m,y) = 0 & on \;\Gamma_{T},\\
\vspace*{0.6em}\left(\dfrac{\partial^2}{\partial x^2}+\dfrac{\partial^2}{\partial y^2}\right)\left( \hat{\xi}_1(x_m)-h(x) \right) = - \mathcal{KC} \tilde{e}(x_m) & on \;\Gamma_{B},\\
\tilde{e}(x_m,y)\mid_{initial} = \tilde{e}(x_{m-1},y) & in \;\bar{\Omega}.
\end{cases}
\end{equation}
Here $h(x)$ is the true analytical solution on $\Gamma_{B}$ using consistent Cauchy data. Further using the assumption that Laplace equation is valid on $\Gamma_{B}$, the above system of error dynamic equation can also be written in an equivalent form as,\\~\\
For $m\geq 1$, find $\tilde{e}(x_m,y)\in \bar{\Omega}$:
\begin{equation}
\label{error2}
\begin{cases}
\vspace*{0.6em}\dfrac{\partial}{\partial x} \tilde{e}(x_m,y) = (\mathcal{A-KC})\tilde{e}(x_m,y) & in \;\bar{\Omega}\backslash \Gamma_{T},\\
%e = 0 & on \;\Gamma_{out},\\
\vspace*{0.6em}\dfrac{\partial}{\partial y}\tilde{e}(x_m,y) = 0 & on \; \Gamma_{T},\\
\tilde{e}(x_m,y)\mid_{initial} = \tilde{e}(x_{m-1},y) & in \;\bar{\Omega}.
\end{cases}
\end{equation}
First equation in (\ref{error2}) is a system of ODEs in variable $x$ and solution to this system has to do with the exponential or the semigroup generated by operator matrix $\mathcal{A-KC}$. Let us denote this semigroup with $\mathbb{S}$. Then solution to above system of ODEs can be written as,
%First equation in (\ref{error2}) is an ODE in variable $x$ and solution to this ODE problem is a decaying exponential as follows,
\begin{equation}
%\tilde{e}(x_m,.) = \mathrm{e}^{(\mathcal{A-KC}) x_m}\tilde{e}(x_0,.) \quad\quad m\geq 1,
\tilde{e}(x_m,.) =  \mathbb{S}_{x_m}\tilde{e}(x_0,.) \quad\quad m\geq 1,
\end{equation}
for a particular iteration index $m$, $x_m$ is $x\in[0,a)$ over $m$-th iteration. Given that under certain conditions semigroup generated by $\mathcal{A}$ is strongly continuous, the observer gain $\mathcal{K}$ can be chosen in a way that $\mathcal{A-KC}$ is dissipative. Then $\mathbb{S}_{x_m}$ will decay exponentially and state estimation error $\tilde{e}(x_m,.)$, for a number of iterations over the whole domain, asymptotically converges to zero for any bounded initial value of $\tilde{e}(x_0,.)$.
%
%However given that under certain conditions semigroup generated by $\mathcal{A}$ is strongly continuous, the semigroup generated by $\mathcal{A-KC}$ can be made exponentially decaying by making operator product $\mathcal{KC}$ positive definite and enough large.
\end{proof}

\subsection{Existence of observer gain $\mathcal{K}$}
Let Hilbert space $X$ as given in equations (\ref{space-x}) and (\ref{norm}) be composed of two mutually exclusive parts as,
\begin{equation}
	X = X_1 \oplus X_2,
\end{equation}
where $X_1$ satisfy conditions as stated in theorem \ref{theorem1} such that $\mathcal{A}$ forms a strongly continuous semigroup and $X_1$ and $X_2$ both make the full space $X$. Following theorem provides conditions on the existence of operator gain $\mathcal{K}$.
\begin{th1}
\label{theorem2}
Under conditions as stated in theorem \ref{theorem1}, let $\mathcal{A}$ as given in equation (\ref{operator-a}) be the generator of a strongly continuous semigroup, $\mathcal{C}\in L(X_1,Y)$ be an observation operator and $Y=\mathbb{R}$), then the following assertions are equivalent.
\begin{enumerate}
\item There exists a positive definite self-adjoint operator product $\mathcal{KC}\in L(X_1)$ where $\mathcal{K}\in L(Y,X_1)$ such that $\mathcal{A-KC}$ generates a maximally dissipative semigroup.
\item There exists arbitrarily small $\epsilon>0$ such that if $\left\| \bar{x}-x \right\|<\epsilon$ then pair $\mathcal{(C,A)}$ is exactly observable in time-like interval $\epsilon$.
\end{enumerate}

\begin{proof}
Given self-adjoint positive definite operator product $\mathcal{KC}\in L(X_1)$, let us denote by $\mathbb{S}$ and $\mathbb{T}$ the semigroups generated by $\mathcal{A-KC}$ and $\mathcal{A}$ respectively.\\~\\
$\boxed{1 \Rightarrow 2}$:\\~\\
Assume $\mathbb{S}$ is dissipative, let us show the observability inequality, that is, there exists $\bar{x}, k>0$ such that,
\begin{equation}\label{eq25}
\int_0^{\bar{x}} \left\| \mathcal{C}\mathbb{T}_x e_0 \right\|^2 dx \geq k_{\bar{x}}^2 \left\| e_0 \right\|^2 \quad \forall e_0\in X_1,
\end{equation}
$\mathcal{A}$ is densely defined so the above inequality is enough to prove exact observability. Given $e_0 \in D(\mathcal{A})$, $e(x)=\mathbb{S}_x e_0$ presents the unique solution of,
\begin{equation}
\label{error_eq}
\begin{cases}
\dfrac{\partial e}{\partial x} = \mathcal{(A-KC)} e(x),\\
e(0) = e_0.
\end{cases}
\end{equation}
Multiplying first equation in (\ref{error_eq}) by $e(x)$,
\begin{eqnarray}
\label{eq37}\dfrac{1}{2} \dfrac{d}{dx}\left\| e(x) \right\|^2 &=& Re \left\langle \dfrac{\partial e}{\partial x}, e(x) \right\rangle,\nonumber\\
&=& Re \left\langle (\mathcal{A-KC})e(x), e(x) \right\rangle,
\end{eqnarray}
in this part we assume that $\mathcal{A-KC}$ is m-dissipative,
% and further integration from $0$ to $\bar{x}$ gives,\\
%\begin{eqnarray}
%\dfrac{1}{2} \int_0^{\bar{x}} \dfrac{d}{dx} \left\| e(x) \right\|^2 dx & = & -\int_0^{\bar{x}} \alpha^2 dx,\nonumber\\
%\label{eq33}
%\left\| e_0 \right\|^2 - \left\| e(\bar{x})\right\|^2 & = & 2\int_0^{\bar{x}} \alpha^2 dx.
%\end{eqnarray}
\begin{eqnarray}
\label{eq32}
\dfrac{d}{dx} \left\| e(x) \right\|^2 & \leq & 0,
%\\
%\label{eq33}
%\left\| e(\bar{x}) \right\|^2 - \left\| e_0\right\|^2 & \leq & 0.
\end{eqnarray}
%where $Re \left\langle (\mathcal{A-KC})e(x), e(x) \right\rangle = -\alpha^2$, as $\mathcal{A-KC}$ generates a dissipative semigroup and $\alpha^2 >0$. 
Let $e(x) = \gamma(x) + \zeta(x)$ such that $\gamma=\mathbb{T}_x e_0$ is the solution of,
% $\|e\|^2>\left\{\|\gamma\|^2;\|\zeta\|^2\right\}$ and $\left\langle \gamma, \zeta\right\rangle>0$, where $\gamma=\mathbb{T}_x e_0$ is the solution of,
\begin{equation}
\label{gamma}
\begin{cases}
\dfrac{\partial \gamma}{\partial x} = \mathcal{A} \gamma(x),\\
\gamma(0) = e_0,
\end{cases}
\end{equation}
and $\zeta$ is the solution of,
\begin{equation}
\label{gamma_sys}
\begin{cases}
\dfrac{\partial \zeta}{\partial x} = \mathcal{A}\zeta(x) - \mathcal{KC} e(x),\\
\zeta(0) = 0.
\end{cases}
\end{equation}
Further we have that $\mathcal{KC}$ is positive definite,
\begin{eqnarray}
	\label{eq40} 0 \leq Re\left\langle \mathcal{KC}\gamma, \mathcal{KC}\gamma \right\rangle \leq  \left\| \mathcal{KC}\gamma(x) \right\|_{X_1}^2.
\end{eqnarray}
Combining equations (\ref{eq32}) and (\ref{eq40}),
\begin{eqnarray}
	\dfrac{d}{dx} \left\| e(x) \right\|^2 & \leq & \left\| \mathcal{KC}\gamma(x) \right\|_{X_1}^2,
\end{eqnarray}
integrating both sides,
\begin{eqnarray}
	\dfrac{1}{2} \int_{0}^{\bar{x}} \dfrac{d}{dx} \left\| e(x) \right\|^2 dx & \leq & \dfrac{1}{2} \int_{0}^{\bar{x}} \left\| \mathcal{KC}\gamma(x) \right\|_{X_1}^2 dx,\nonumber\\
	2\left(\left\| e(\bar{x}) \right\|^2 - \left\| e_0\right\|^2 \right) & \leq & \int_{0}^{\bar{x}} \left\| \mathcal{KC}\gamma(x) \right\|_{X_1}^2 dx,\nonumber\\
%	2 \left\| e(\bar{x}) \right\|_{X_1}^2 & \leq & \int_{0}^{\bar{x}} \left\| \mathcal{KC}\gamma(x) \right\|_{X_1}^2 dx,\nonumber\\
\end{eqnarray}
finally we have,
\begin{eqnarray}
	k \left\| e(\bar{x}) \right\|_{X_1}^2 & \leq & \int_{0}^{\bar{x}} \left\| \mathcal{C}\gamma(x) \right\|_{Y}^2 dx,
\end{eqnarray}
where $k=\dfrac{2}{\| \mathcal{K} \|^2} >0$ is independent of $\bar{x}$. For arbitrarily small time-like interval $\bar{x}-0=\bar{x}=\epsilon$, above inequality is same as observability inequality (\ref{eq25}).
\\~\\
$\boxed{2 \Rightarrow 1}$:\\~\\
We have $e(x) = \gamma(x) + \zeta(x)$, where $\gamma$ is the solution of open loop system (\ref{gamma}) and $e(x)$ is the solution of closed loop feedback system (\ref{error_eq}), we have that,
\begin{eqnarray}
\left\langle(\mathcal{A-KC})\gamma(x),\gamma(x)\right\rangle &\geq & \left\langle(\mathcal{A-KC}) e(x), e(x)\right\rangle,
\end{eqnarray}
multiplying with $-1$ and using inequality $(\ref{eq37})$,
\begin{equation}
	\left\langle(\mathcal{KC})\gamma(x),\gamma(x)\right\rangle \leq -\dfrac{1}{2} \dfrac{d}{dx}\left\| e(x) \right\|^2 + \left\langle(\mathcal{A})\gamma(x),\gamma(x)\right\rangle,
\end{equation}
integrating both sides from $0$ to $\bar{x}$ and a positive $\alpha>1$ such that,
\begin{eqnarray}
	\label{47}\int_{0}^{\bar{x}} \left\langle(\mathcal{KC})\gamma(x),\gamma(x)\right\rangle dx &\leq& \alpha\left(\left\| e_0 \right\|^2 - \left\| e(\bar{x})\right\|^2\right) \nonumber \\ & &+ \int_0^{\bar{x}} \left\langle \mathcal{A}\gamma(x), \gamma(x) \right\rangle dx,
\end{eqnarray}
Further we can write,
\begin{eqnarray}
Re \left\langle \mathcal{A}\gamma, \gamma \right\rangle &=& Re\left\langle \dfrac{\partial \gamma}{\partial x}, \gamma \right\rangle, \nonumber \\ \nonumber
 %&=& \dfrac{1}{2} \dfrac{d}{dx} \left\| \gamma(x) \right\|^2,\\ \nonumber
 &=& \dfrac{1}{2} \dfrac{d}{dx} \left\| e(x) - \zeta(x) \right\|^2,\\
 \nonumber & \leq & \dfrac{1}{2} \dfrac{d}{dx} \left\| e(x) \right\|^2 + \dfrac{1}{2} \dfrac{d}{dx} \left\| \zeta(x) \right\|^2,\\
 & \leq & \dfrac{d}{dx} \left\| e(x) \right\|^2 + \dfrac{1}{2} \dfrac{d}{dx} \left\| \gamma(x) \right\|^2,
\end{eqnarray}
given that $\mathcal{A}$ generates a strongly continuous semigroup, there exists $\beta>1$ such that,
\begin{equation}
	Re \left\langle \mathcal{A}\gamma, \gamma \right\rangle \leq \beta \dfrac{d}{dx} \left\| e(x) \right\|^2,
\end{equation}
integrating both sides,
\begin{equation}
	\int_{0}^{\bar{x}} Re \left\langle \mathcal{A}\gamma, \gamma \right\rangle dx \leq -\beta \left( \left\| e_0 \right\|^2 - \left\| e(\bar{x}) \right\|^2 \right),
\end{equation}
combining above inequality with (\ref{47}),
\begin{equation}
	\label{47}\int_{0}^{\bar{x}} \left\langle(\mathcal{KC})\gamma(x),\gamma(x)\right\rangle dx \leq (\alpha-\beta)\left(\left\| e_0 \right\|^2 - \left\| e(\bar{x})\right\|^2\right),
\end{equation}
$\alpha$ and $\beta$ can be chosen appropriately large, let us take $\alpha-\beta=2$, we have
\begin{eqnarray}
\label{eq47}
\int_0^{\bar{x}} \left\langle \mathcal{KC}\gamma(x),\gamma(x) \right\rangle dx &\leq & 2\left( \left\| e_0 \right\|^2 - \left\| e(\bar{x}) \right\|^2 \right), \nonumber \\
\int_0^{\bar{x}} \left\langle \mathcal{C}\gamma(x), \mathcal{C}\gamma(x) \right\rangle dx &\leq & \dfrac{2}{ \| \mathcal{K} \|} \left( \left\| e_0 \right\|^2 - \left\| e(\bar{x}) \right\|^2 \right), \nonumber \\
\int_0^{\bar{x}} \left\| \mathcal{C}\gamma(x) \right\|_Y^2 dx &\leq & \dfrac{2}{\| \mathcal{K} \|} \left( \left\| e_0 \right\|^2 - \left\| e(\bar{x}) \right\|^2 \right),
\end{eqnarray}
%above inequality is from the fact that $\left\langle \mathcal{KC}\gamma,\gamma \right\rangle > 0$ and $\mathcal{C}$ is a boundary observation operator. 
now using observability inequality (\ref{eq25}) we have,
\begin{eqnarray}
k^2 \left\| e_0\right\|^2 &\leq & \dfrac{2}{ \| \mathcal{K} \|} \left( \left\| e_0 \right\|^2 - \left\| e(\bar{x}) \right\|^2 \right),\nonumber\\
\left\| e(\bar{x}) \right\|^2 &\leq & \left( 1-\dfrac{1}{\| \mathcal{K}\|}\right)\left\| e_0 \right\|^2
, %\nonumber \\
%\left\| e(\bar{x}) \right\|^2 &\leq & \left( 1-\left( \dfrac{1-M^2 e^{-2\omega \bar{x}}}{\|\mathcal{K}\|} \right) \right) \left\| e_0 \right\|^2,
\end{eqnarray}
%where $\left( 1-\left( \dfrac{1-M^2 e^{-2\omega \bar{x}}}{\|\mathcal{K}\|} \right) \right) < 1$ because from observability inequality $M^2 e^{-2\omega \bar{x}}<1$. Hence semigroup generated by $\mathcal{A-KC}$ is maximally dissipative.
where $\left( 1-\dfrac{1}{\| \mathcal{K}\|}\right)<1$ if $\| \mathcal{K}\| > 1$.

\end{proof}
\end{th1}
In section \ref{section_numerical}, the observer is implemented numerically using fictitious points on the estimated solution boundary. Numerical results are presented in the last section.

\section{Numerical implementation and results}\label{section_numerical}
First order state equation given in (\ref{inverse2}) can be discretized in variable $x$ using Forward Euler as follows,
\begin{eqnarray}
\label{discrete1}
\dot{\xi} &=& \mathcal{A}\xi, \nonumber\\
\frac{\xi^{n+1}-\xi^n}{\triangle x} &=& \mathcal{A}\xi^n, \nonumber\\
\xi^{n+1} &=& (I+(\triangle x)\mathcal{A})\xi^{n},
\end{eqnarray}
here $I$ is the identity matrix, $n$ is the discrete index for variable $x$ and $\triangle x$ is the step size along $x$ after discretization. Further equation (\ref{discrete1}) is discretized in variable $y$ using second order accurate centered finite difference schemes to discretize the first and second order derivative terms. Cauchy data is available on the top boundary however for the bottom boundary $\Gamma_{B}$ there's no data available and we assume that Laplace equation is valid on this boundary as given in problem (\ref{observer1}). Numerically this condition can be implemented using fictitious points along the inner boundary as explained in following section. 
%\cite{c12}
%This assumption is theoretically proved to be valid as operator $\mathcal{A}$ is dissipative as proved in theorem \ref{theorem1} and state estimation error converges to zero assymptotically if and only if observer gain $\mathcal{K}$ is designed in such a way that $\mathcal{KC}(\hat{\xi}-\xi)\approx 0$ after a number of observer iterations.

\subsection{Boundary condition on $\Gamma_{B}$}
As stated above, first order state equation can be thought of as an ODE with respect to variable $x$. Solution of this ODE is the state $\xi$ over the whole vertical line, that is, $\left(y\mid_{\Gamma_{B}},y\mid_{\Gamma_{T}}\right)$. This can be thought of as 2D Laplace equation has been split into a series of 1D state equations. To solve this 1D state equation in variable $x$, initial condition over the whole interval $\left(y\mid_{\Gamma_{B}},y\mid_{\Gamma_{T}}\right)$ and boundary conditions on $\Gamma_{B}$ and $\Gamma_{T}$ are required. Any initial guess can be chosen as $(\mathcal{A-KC})$ will be dissipative and any initial guess dies out. Two bounday conditions on $\Gamma_{T}$ are available, that is the measurement data on $\Gamma_{T}$ and available Neumann boundary data. However on $\Gamma_{B}$ it is assumed that Laplace equation is satisfied. That is,
\begin{equation}
\label{boundaryin}
\frac{\partial^2 u}{\partial x^2}= -\frac{\partial^2 u}{\partial y^2} \;\;\; on \;\Gamma_{B}.
\end{equation}
Equation (\ref{boundaryin}) contains second order derivative in variable $y$. To discretize this second derivative using second order accurate centered finite difference discretization scheme on $\Gamma_{B}$, there needs to be a fictitious point \cite{p1} further outside the boundary $\Gamma_{B}$ as shown in figure \ref{fictitious_fig}.
\vspace{1em}
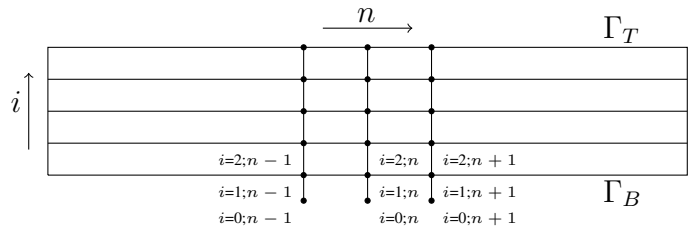
\begin{figure}[thpb]
			\centering
%	\framebox{\parbox{4in}
%	{
	\begin{center}
		\begin{tikzpicture}[scale=0.85]
			\begin{scope}
				\draw (0,0) -- (10,0) -- (10,2) -- (0,2) -- (0,0);
				\draw (5,-0.4) -- (5,2);
				\draw (4,-0.4) -- (4,2);
				\draw (6,-0.4) -- (6,2);
				\draw (0,0.5) -- (10,0.5);
				\draw (0,1) -- (10,1);
				\draw (0,1.5) -- (10,1.5);
				\draw[fill] (5,-0.4) circle [radius=0.04];
				\node [below right, black] at (5,-0.4) {\tiny{$i$=0;$n$}};
				\draw[fill] (5,0) circle [radius=0.04];
				\node [below right, black] at (5,0) {\tiny{$i$=1;$n$}};
				\draw[fill] (5,0.5) circle [radius=0.04];			
				\node [below right, black] at (5,0.5) {\tiny{$i$=2;$n$}};	
				\draw[fill] (4,-0.4) circle [radius=0.04];			
				\node [below left, black] at (4,-0.4) {\tiny{$i$=0;$n-1$}};
				\draw[fill] (4,0) circle [radius=0.04];			
				\node [below left, black] at (4,0) {\tiny{$i$=1;$n-1$}};
				\draw[fill] (4,0.5) circle [radius=0.04];			
				\node [below left, black] at (4,0.5) {\tiny{$i$=2;$n-1$}};
				\draw[fill] (6,-0.4) circle [radius=0.04];			
				\node [below right, black] at (6,-0.4) {\tiny{$i$=0;$n+1$}};
				\draw[fill] (6,0) circle [radius=0.04];			
				\node [below right, black] at (6,0) {\tiny{$i$=1;$n+1$}};
				\draw[fill] (6,0.5) circle [radius=0.04];			
				\node [below right, black] at (6,0.5) {\tiny{$i$=2;$n+1$}};
				\draw[fill] (6,1.0) circle [radius=0.04];
				\draw[fill] (6,1.5) circle [radius=0.04];						
				\draw[fill] (6,2.0) circle [radius=0.04];		
				\draw[fill] (5,1.0) circle [radius=0.04];
				\draw[fill] (5,1.5) circle [radius=0.04];						
				\draw[fill] (5,2.0) circle [radius=0.04];	
				\draw[fill] (4,1.0) circle [radius=0.04];
				\draw[fill] (4,1.5) circle [radius=0.04];						
				\draw[fill] (4,2.0) circle [radius=0.04];
				\draw[->] (4.3,2.3) -- (5.7,2.3);
				\put (117,58){$n$};
				\draw[->] (-0.3,0.4) -- (-0.3,1.6);
				\put (-14, 24){$i$};
				\put (210,-10){$\Gamma_{B}$};
				\put (210,52){$\Gamma_{T}$};				
				%\draw[fill=light-gray] (2,0) rectangle (6cm,2.5cm);				
				%\draw (-1,-2) rectangle (5,2.5);
				%\put(70,30){$\Omega=(0,a)\times (0,b)$};
				%\put(100,-12){$\Gamma_{B}$};
				%\put(100,77){$\Gamma_{T}$};
				%\put(42,30){${\Gamma_{L}}$};
				%\put(35,-12){$(0,0)$};
				%\put(35,75){$(0,b)$};
				%\put(165,-12){$(a,0)$};
				%\put(165,75){$(a,b)$};
				%\put(173,30){${\Gamma_{R}}$};
			\end{scope}
		\end{tikzpicture}
	\end{center}
%	}
	\caption{Domain $\Omega$ after discretization and fictitious points outside $\Gamma_{B}$, index $i=0$ represents fictitious points.}
	\label{fictitious_fig}
\end{figure}

Second equation in (\ref{discrete1}), after full discretization, can be written as,
\begin{equation}
\label{fulldiscret}
\dfrac{(\xi_2)_i^{n+1}-(\xi_2)_i^{n}}{\triangle x}  =- \dfrac{(\xi_1)_{i+1}^{n}-2(\xi_1)_{i}^{n}+(\xi_1)_{i-1}^{n}}{(\triangle y)^2},
\end{equation}
here $i$ is the discrete index and $(\triangle y)$ is the step size along variable $y$ such that $i=1$ on the bottom boundary $\Gamma_{B}$. Equation (\ref{fulldiscret}) on the bottom boundary $\Gamma_{B}$ can be written as,
\begin{equation}
\dfrac{(\xi_2)_1^{n+1}-(\xi_2)_1^{n}}{\triangle x}  =- \dfrac{(\xi_1)_{2}^{n}-2(\xi_1)_{1}^{n}+(\xi_1)_{0}^{n}}{(\triangle y)^2},
\end{equation}
index $i=0$ represents fictitious point and taking out this fictitious point gives,
\begin{equation}
\label{fictitious}
%\begin{split}
(\xi_1)_{0}^{n} = 2(\xi_1)_1^n-(\xi_1)_2^n -\dfrac{(\triangle y)^2}{\triangle x} \left\{ (\xi_2)_1^{n+1}-(\xi_2)_1^n \right\}.
%\end{split}
\end{equation}
$(\xi_1)_1^n, (\xi_1)_2^n, (\xi_2)_1^n$ and $(\xi_2)_1^{n+1}$ are given by the initial guess of the states over the whole domain. The algorithm is run for a number of iterations along $x$ by using the solution of the previous iteration as a guess for the next until the final convergence is achieved. In the following subsection observer is presented in semi-discrete form and fictitious points method is used to tackle the boundary condition on $\Gamma_{B}$.

\subsection{Observer in semi-discrete form}
%In control theory, state observer is a system that provides an estimate of the internal states of a dynamic system from the measurements of inputs and outputs. In this work steady state heat conduction problem is posed as dynamical problem, where variable $\theta$ plays the role of time component like other dynamical systems. \\~\\
In the following, state observer presented in system of equations (\ref{observer1}) is discretized only in variable $x$ for simplicity.
\begin{eqnarray}
\label{observer2}
\begin{cases}
\hat{\xi}^{n+1,m}=(I+(\triangle x)\mathcal{A})\hat{\xi}^{n,m}-\mathcal{KC}(\hat{\xi}^{n,m}-\xi^{n}) & in \;\Omega,\\
%{\hat{\xi}}^n = \xi_{init} & in \;\bar{\Omega}\backslash\Gamma_{out},\\
%{\hat{\xi}_{1}}^n=h_{init}^n(r,\theta) & on \;\Gamma_{in},\\
%\hat{\xi}_1^n=\hat{f}^n(\theta) & on \;\Gamma_{out},\\
\dfrac{\partial}{\partial y}\hat{\xi}_{1}^{n,m}=g^{n}(x) & on \;\Gamma_{T},\\
\dfrac{1}{2\triangle x} \left( \hat{\xi}_1^{n+1,m}-\hat{\xi}_1^{n-1,m}\right) = -\dfrac{\partial^2}{\partial y^2} \hat{\xi}_1^{n,m} \\ \quad\quad\quad\quad\quad\quad\quad\quad\quad\quad\quad\quad -\mathcal{KC} (\hat{\xi}^{n,m}-\xi^n) & on \;\Gamma_{B},\\
\hat{\xi}^{n,m}\mid_{initial} = \hat{\xi}^{n,m-1} & in \;\bar{\Omega},
\end{cases}
\end{eqnarray}
again here $\hat{}$ represents estimated quantity and $m$ is the index of iteration over the rectangular domain. $\hat{\xi}^{m,n}\mid_{initial}$ represents estimate for particular value of index $n$ at the start of $m^{th}$ iteration. Algorithm starts at $m=1$ and index $m=0$ represents the raw data over the whole mesh before start of the algorithm. At the start of the algorithm, any initial guess can be chosen over the whole domain and then solution of the previous iteration is used as a guess for subsequent iterations. Fictitious points are computed using formula given in (\ref{fictitious}) and are used in third equation in (\ref{observer2}) to discretize the second order derivative with respect to $y$ on $\Gamma_{B}$. Important point to note here is that true Neumann boundary condition is applied on the outer boundary and operators $\mathcal{A,K}$ and $\mathcal{C}$ are continuous in variable $y$.
%\subsection{Improving convergence by iteration}
%As the name suggests, space iterative algorithm is run for a number of iterations over the whole annulus domain in the direction of variable $\theta$. Initial guess at the start of first iteration is given by $\hat{\xi}\mid_{\bar{\Omega}}$, whereas for each subsequent iteration result of the previous iteration is used as a guess over $\bar{\Omega}$. We can choose any initial guess and results for number of iterations required to achieve convergence for different initial guess is given in section \ref{5}.

\subsection{Algorithm step-by-step}
\begin{itemize}
\item \textit{Step 1}: Initialize mesh over the whole domain $\bar{\Omega}$ with $\hat{\xi}^{m=0}=\xi_0$.
\item \textit{Step 2}: For $m=1$, start at a particular value of $x$,
\begin{itemize} 
\item Compute the fictitious point value $\left( \xi_1 \right)_0^{n, m=1}$ for particular value of $n$ using equation (\ref{fictitious}).
\item Solve system of equations (\ref{observer2}) to find estimate $\hat{\xi}^{n+1, m}$ over a particular vertical line.
\item Repeat the process of finding fictitious point from equation (\ref{fictitious}) and solving system of equations (\ref{observer2}) for all $n$ until one iteration on interval of length $a$ on the rectangular domain shown in Figure.~\ref{domain} is complete.
\end{itemize}
\item \textit{Step 3}: Repeat \textit{Step 2} for $m\geq 2$ using result of $(m-1)$th iteration as a guess for $m$th iteration until convergence is achieved. That is, $\| \xi_1 - \hat{\xi}^m_1 \|_{\Gamma_{out}}<\epsilon$.
\end{itemize}
\subsection{State estimation error and computation of observer gain}
State error boundary value problem in semi-discrete form can be written as,\\
For $m\geq 1$, find $e^{n,m}=(\hat{\xi}^{n,m}-\xi^n)\in\bar{\Omega}$:
\begin{equation}
\begin{cases}
e^{n+1,m}= \left(I+(\triangle x)\mathcal{A} -\mathcal{KC}\right) (\hat{\xi}^{n, m} - \xi^{n}) & in \;\bar{\Omega}\backslash\Gamma_{T},\\
\dfrac{\partial}{\partial y}e_1^{n,m}=\dfrac{\partial}{\partial y} \left( \hat{\xi}_1^{n, m} -\xi_1^{n} \right) = 0 & on \;\Gamma_{T},\\
e^{n,m}\mid_{initial} = e^{n,m-1} & in \;\bar{\Omega}.
\end{cases}
\end{equation}
%\begin{eqnarray}
%\label{state_error_1}
%\hat{\xi}^{n+1}-\xi^{n+1} &=& \left(I+(\triangle \theta)\mathcal{A}\right) (\hat{\xi}^{n}-\xi^{n})-\mathcal{KC}(\hat{\xi}^n-\xi^n) \nonumber, \\
%e^{n+1} &=& \left(I+(\triangle \theta)\mathcal{A}-\mathcal{KC}\right)e^n.
%\end{eqnarray}
Finally the state error difference equation after full discretization can be written as,
\begin{equation}
\label{state_error_2}
\mathtt{e}^{n+1, m} = \left(I+(\triangle x)A-KC\right)\mathtt{e}^{n, m}\;\;\;\;\;\;for\;\;m\geq 1,
\end{equation}
here $A,K$ and $C$ are discrete versions of operators $\mathcal{A,K}$ and $\mathcal{C}$ respectively and $\mathtt{e}$ is the state estimation error after full discretization. Given $(I+(\triangle x) A)$ and observation matrix $C$, gain matrix $K$ can be computed using Ackermann's formula for pole placement in Matlab such that eigenvalues of $(I+(\triangle x)A-KC)$ are inside the unit circle on the complex plane \cite{c13}. 
%A necessary condition for the convergence of observer of this kind is the system observability.
\subsection{Results and simulations}
\label{5}
For all numerical and analytical solutions in this section, a rectangle domain $\Omega=(0,a)\times(0,b)$ with $a=2\pi$ and $b=\frac{1}{2}$ is considered. To validate the observer approach a number of examples are presented as follows.
%Eigenvalues of discrete operator matrix $A$ are close to one another in magnitude. This shows the high sensitivity of the problem. We will see in this section that a Tikhonov kind of regularization can be used to reduce the high sensitivity of the problem.
\subsubsection{Example 1: Homogeneous Neumann side boundaries}
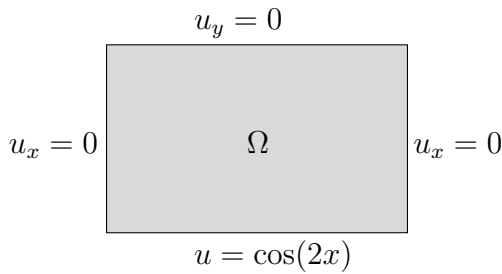
\begin{figure}[thpb]
			\centering
%	\framebox{\parbox{4in}
%	{
	\begin{center}
		\begin{tikzpicture}%[framed]
			\begin{scope}
				\draw[fill=light-gray] (2,0) rectangle (6cm,2.5cm);				
				%\draw (-1,-2) rectangle (5,2.5);
				\put(110,30){$\Omega$};
				\put(90,-12){$u=\cos(2x)$};
				\put(90,77){$u_y=0$};
				\put(20,30){$u_x=0$};
				\put(173,30){$u_x=0$};
			\end{scope}
		\end{tikzpicture}
	\end{center}
%	}
	\caption{Two dimensional rectangle domain with homogeneous Neumann side boundaries, Example 1.}
	\label{example1}
\end{figure}
Consider the boundary value problem in a rectangular domain with homogeneous Neumann side boundaries as shown in Figure~\ref{example1}. This problem can be solved using separation of variables and solution is given as,
\begin{equation}
\label{example1_sol}
u(x,y) = \dfrac{\cosh(4\pi (y-b)/a)}{\cosh(4\pi b/a)}\cos(4\pi x/a),
\end{equation}
To validate observer based approach this analytical solution given in (\ref{example1_sol}) along with homogeneous Neumann boundary condition is used as Cauchy data on the top boundary $\Gamma_T$. Using this Cauchy data state observer algorithm is run for a number of iterations to recover the unknown boundary data on the bottom boundary $\Gamma_B$. Figure~\ref{example1_fig} shows the comparison of exact solution and the one recovered by using Cauchy data on $\Gamma_B$ and observer algorithm.
\begin{figure}[thpb]
      \centering
      \framebox{\parbox{3in}{
      
      \begin{center}
      \includegraphics[scale=0.38]{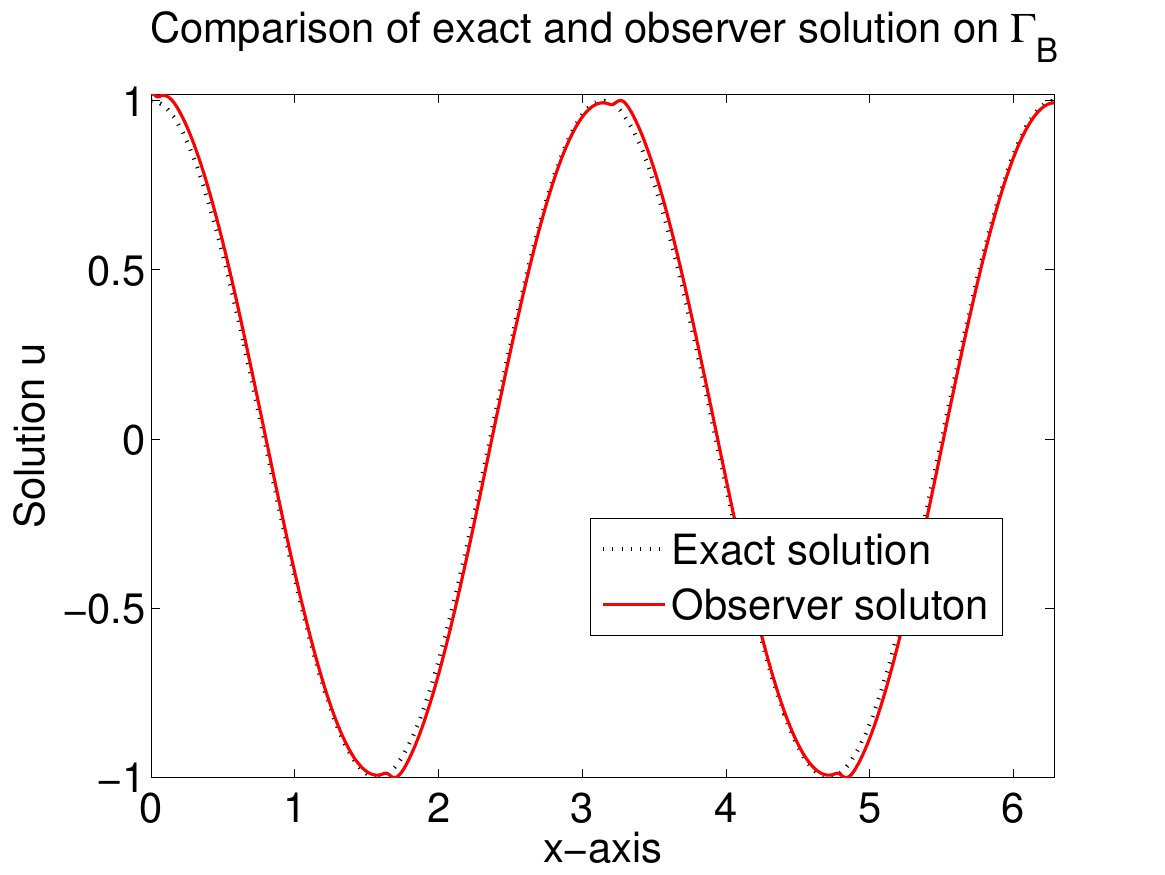}
      \end{center}
      }}
      \caption{Comparison of exact and observer constructed solution on the bottom boundary $\Gamma_B$.}
     \label{example1_fig}
\end{figure}

\subsubsection{Example 2: Homogeneous Dirichlet side boundaries}
\begin{figure}[thpb]
			\centering
%	\framebox{\parbox{4in}
%	{
	\begin{center}
		\begin{tikzpicture}%[framed]
			\begin{scope}
				\draw[fill=light-gray] (2,0) rectangle (6cm,2.5cm);				
				%\draw (-1,-2) rectangle (5,2.5);
				\put(110,30){$\Omega$};
				\put(90,-12){$u=\sin(2x)$};
				\put(90,77){$u_y=0$};
				\put(25,30){$u=0$};
				\put(173,30){$u=0$};
			\end{scope}
		\end{tikzpicture}
	\end{center}
%	}
	\caption{Two dimensional rectangle domain with homogeneous Dirichlet side boundaries, Example 2.}
	\label{example2}
\end{figure}
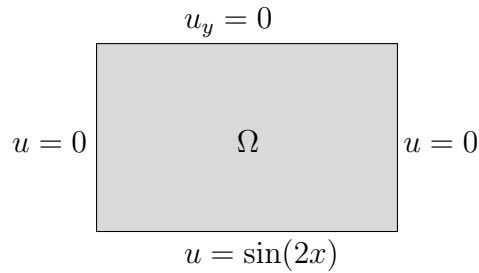
Consider the boundary value problem with homogeneous Dirichlet side boundaries as shown in Figure~\ref{example2}. Analytical solution is given as,
\begin{equation}
\label{example2_sol}
u(x,y) = \dfrac{\cosh(4\pi (y-b)/a)}{\cosh(4\pi b/a)}\sin(4\pi x/a),
\end{equation}
Now using this analytical solution along with homogeneous Neumann boundary condition on the top boundary $\Gamma_B$, observer algorithm is run for a number of iterations to recover the unknown Dirichlet boundary data on $\Gamma_B$. Figure~\ref{example2_fig} shows the comparison of the exact and observer constructed solution on $\Gamma_B$.
\begin{figure}[thpb]
      \centering
      \framebox{\parbox{3in}{
      
      \begin{center}
      \includegraphics[scale=0.38]{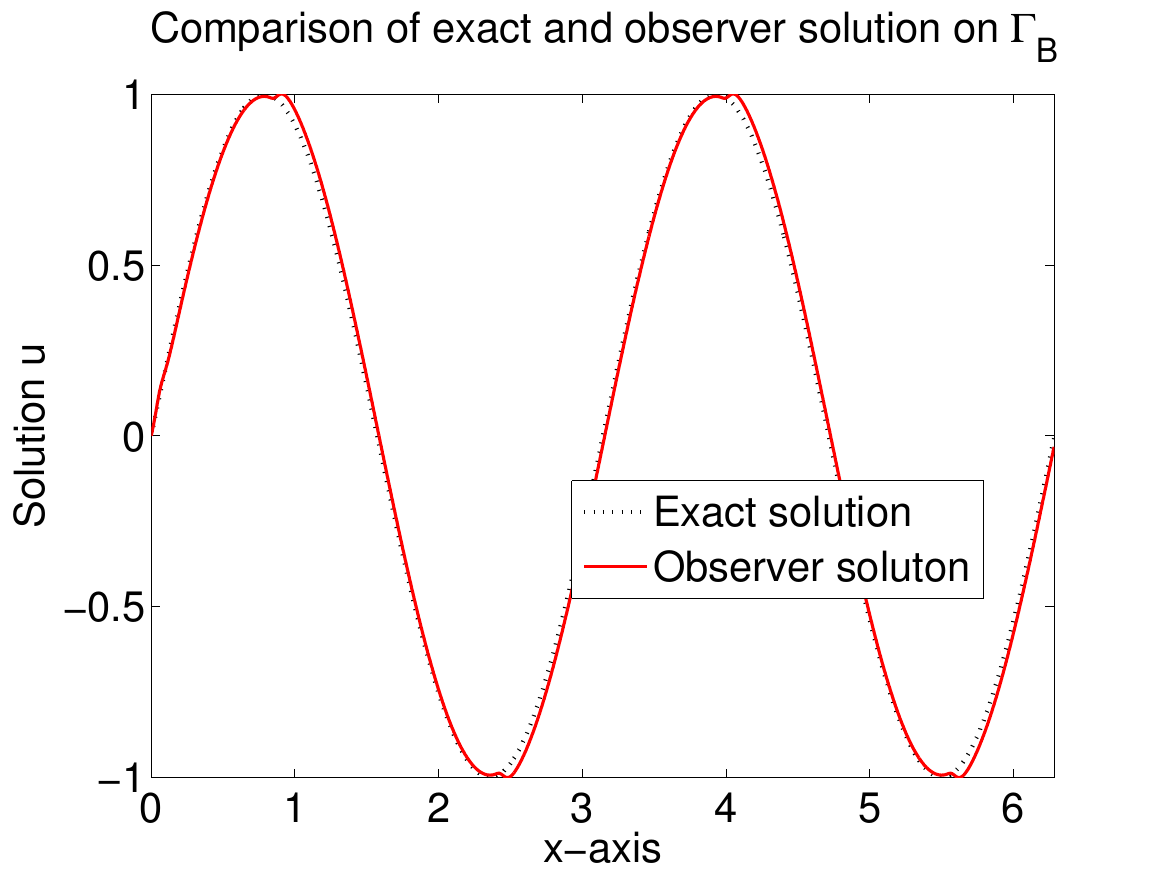}
      \end{center}
      }}
      \caption{Comparison of exact and observer constructed solution on the bottom boundary $\Gamma_B$.}
     \label{example2_fig}
\end{figure}

\subsubsection{Example 3: Linear combinations of example 1 and 2}
It is easy to see that any linear combination of above two example problems can be solved using observer based technique. In other words any Dirichlet boundary data on $\Gamma_B$ that can be represented as a trigonometric Fourier series can be recovered using observer based approach given homogeneous Dirichlet, Neumann or Robin kind of side boundaries. The requirement of such homogeneous side boundaries suggest that there are no active sources on the side boundaries which is indeed the case for many applications like electrocardiography (ECG) where objective is to find heart potential which is deep inside the body from the only available ECG data on a limited part of body torso \cite{c7,c6}. The observer based approach is the optimum technique in cases where there is no information available on the side boundaries. Figure~\ref{example3_fig} compares the exact solutions in different test cases to the one obtained by using observer. Numerical solution was achieved using homogeneous Neumann boundaries on $\Gamma_L, \Gamma_R$ and $\Gamma_T$ and non zero Dirichlet data on $\Gamma_B$. The observer solution was constructed using only the Cauchy data on $\Gamma_T$. 
\begin{figure}[thpb]
	\centering
	\framebox{\parbox{3in}{
	
	\begin{center}
	\includegraphics[scale=0.38]{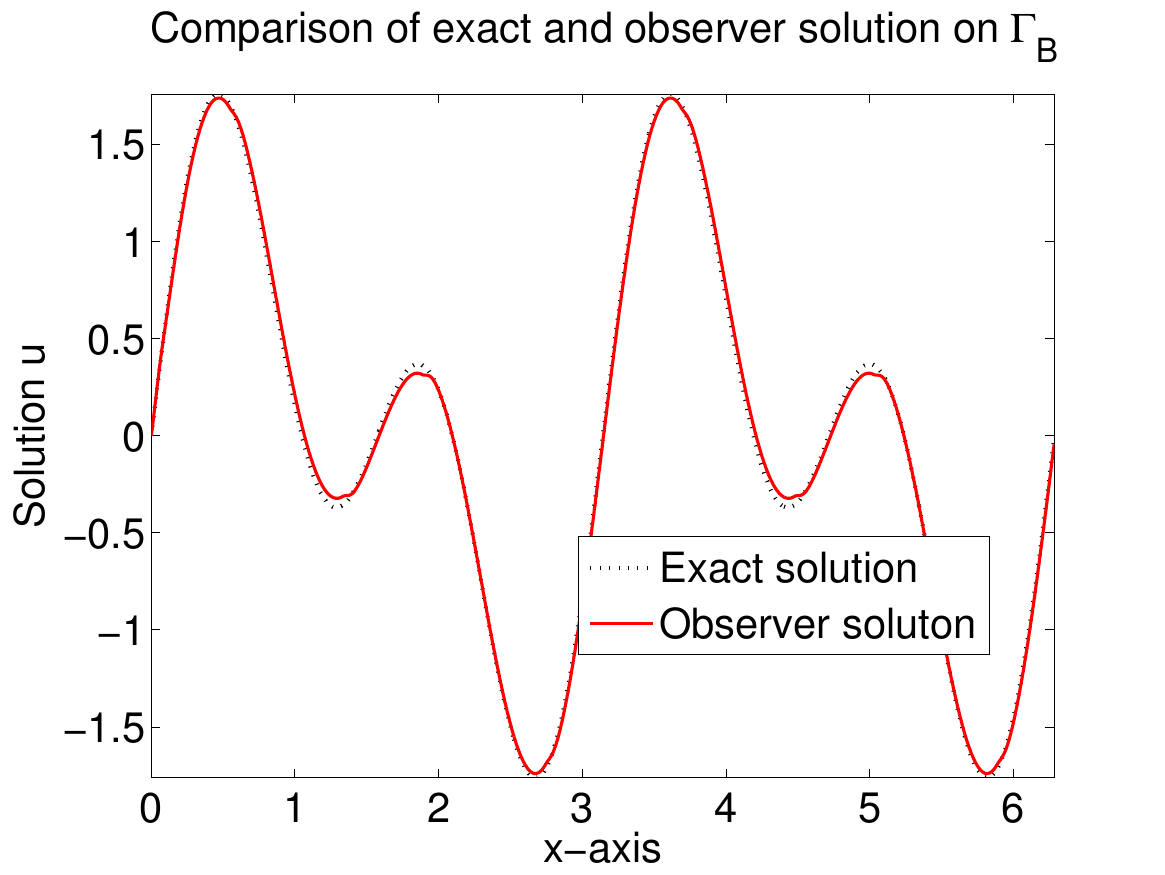} \\~\\ \includegraphics[scale=0.38]{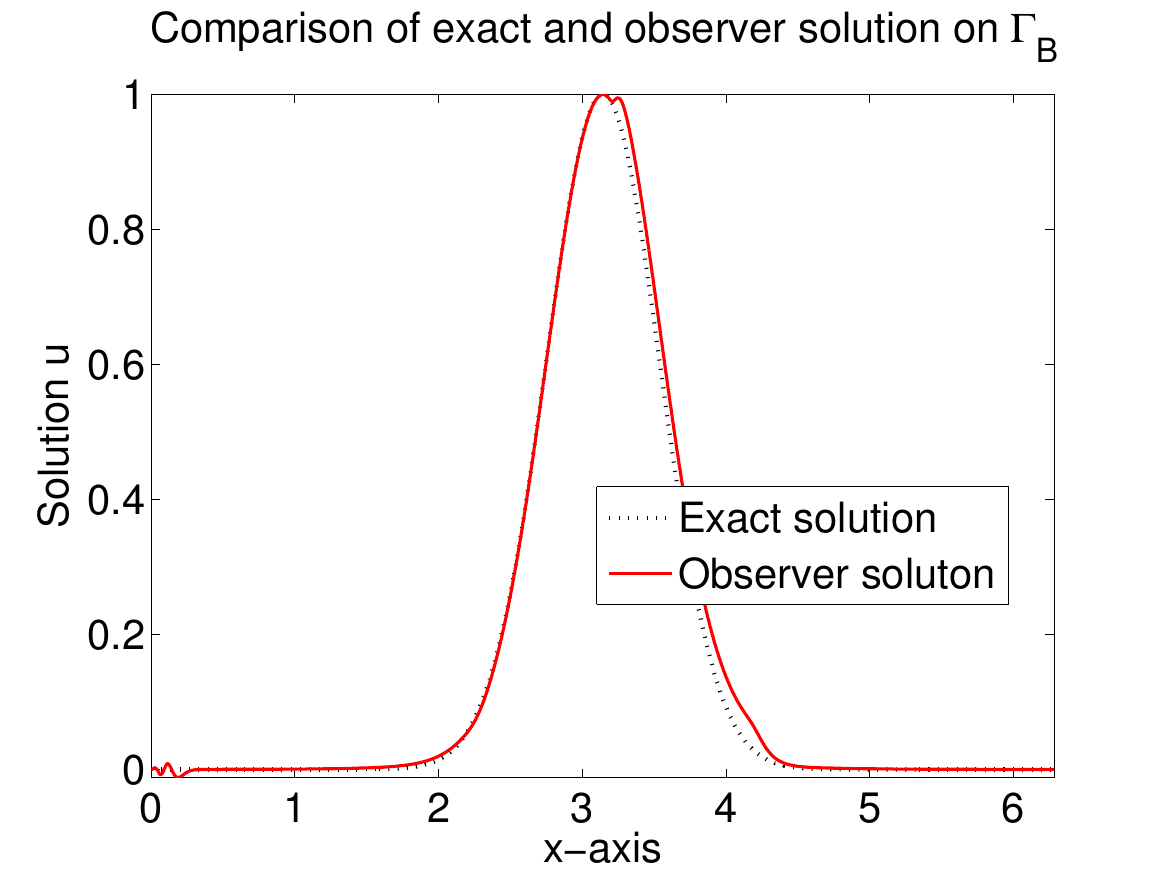}\\~\\
	\includegraphics[scale=0.38]{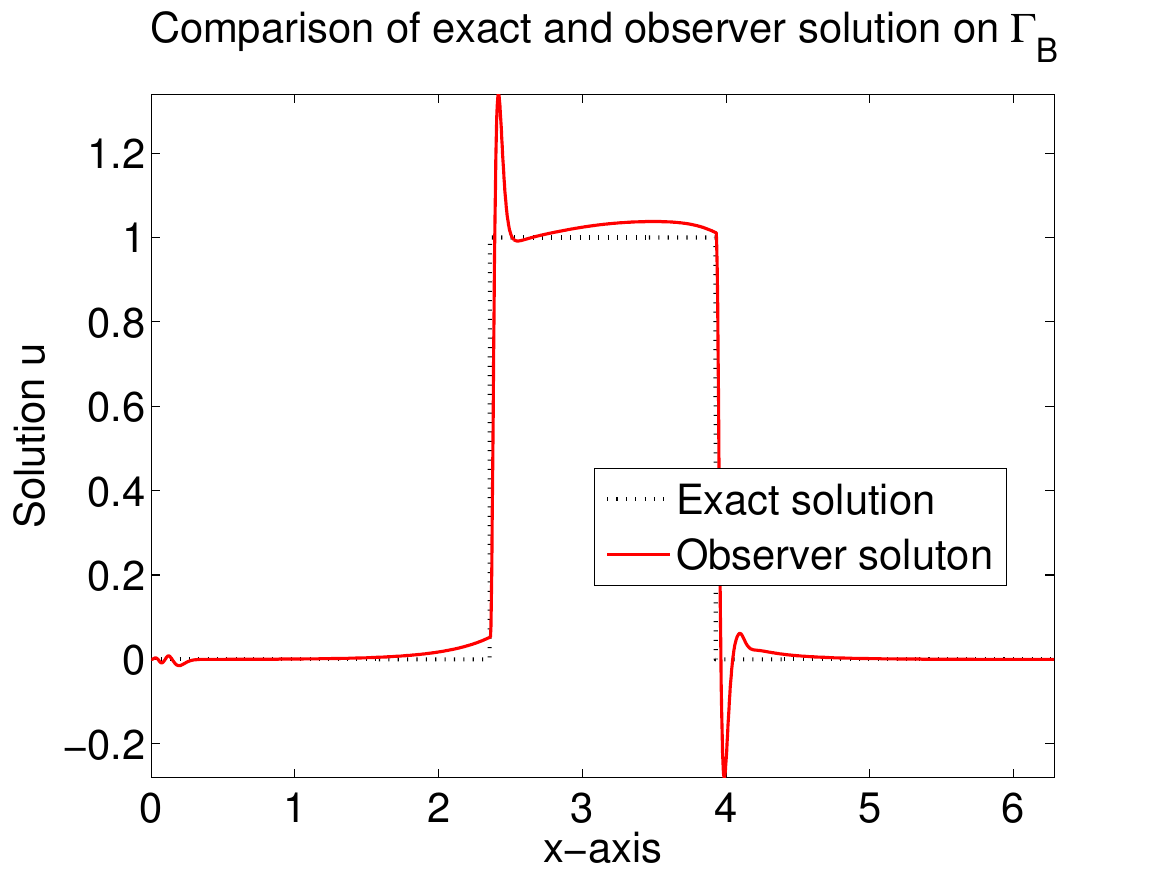} \\~\\ \includegraphics[scale=0.38]{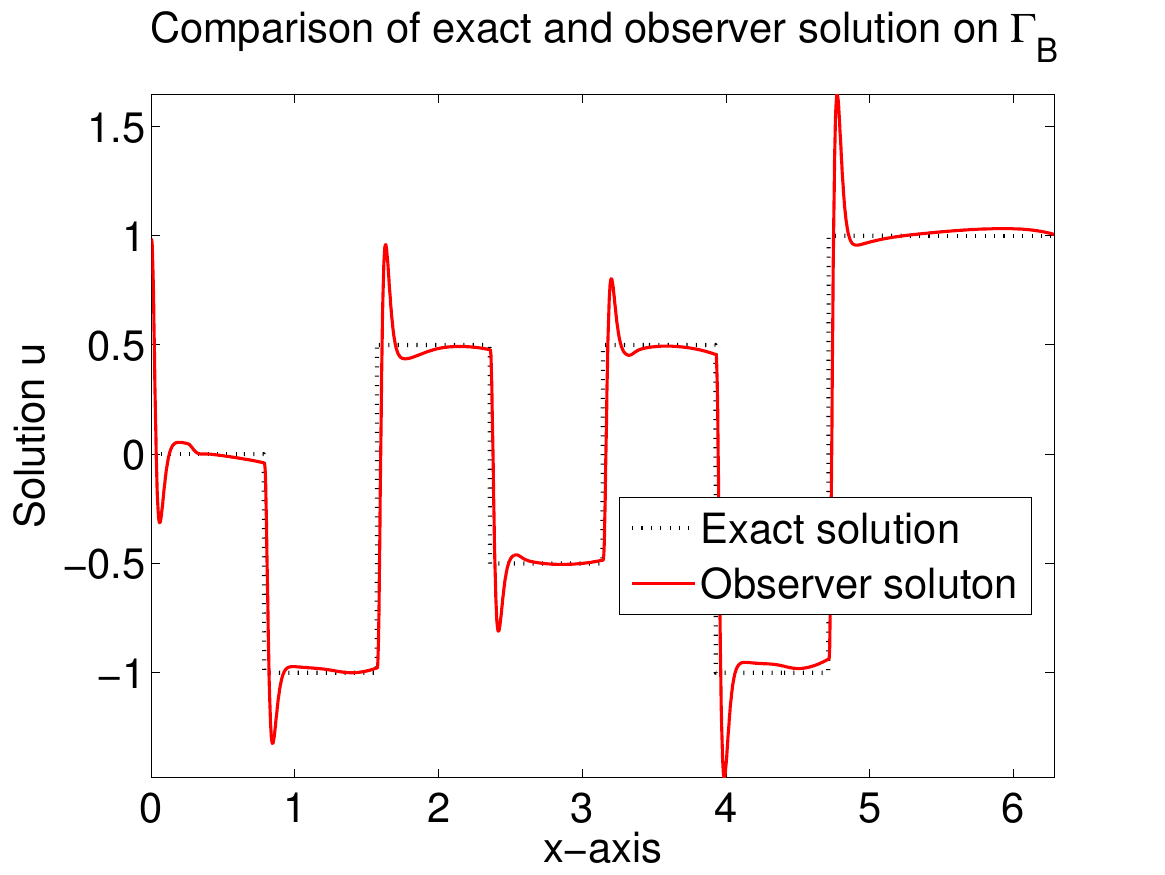}
	\end{center}
	}}
	\caption{Comparison of exact and observer constructed solution on the bottom boundary $\Gamma_B$.}
%	\end{center}
	\label{example3_fig}
\end{figure}

\section{Conclusions}\label{conclusions}
Cauchy problem for Laplace equation is a steady state problem. The design of a dynamical systems' inspired technique like observer for this problem is challenging and the idea to use one of the space variables as a time-like variable has not been considered before.

Different from standard approaches to tackle this problem, an iterative observer is constructed in infinite dimensional setting on a rectangle domain without introducing an extra time variable. Laplace equation is presented as a first order state equation with state operator matrix. % with state operator matrix $\mathcal{A}$. %Semigroup study for this new kind of operator is provided along with system observability result. 
 Conditions for the existence of strongly continuous semigroup generated by this state operator matrix are provided. Further the conditions for the existence of observer gain are detailed. Numerical results are provided for different example test cases. This paper reflects the possibility of considering a steady state problem from a dynamical theory perspective by using one of the space variables as a time. Successful implementation of the algorithm and promising results show a step forward in the direction of using dynamical systems' inspired algorithms to solve steady-state problems modeled by time independent PDEs and without introducing a particular notion of time.

Possible future work includes  extension of the proposed method to three dimensions with more complicated domains using the interesting observability result obtained in this work.

% Observer is designed by writing second order Laplace equation as a first order state equation in one of the space variables. Resulting state operator matrix is studied in Banach space setting using semigroup theory. System comprising of state operator matrix and observation operator is shown to be observable and existence of a gain operator to recover the unknown boundary data is proved. Finally numerical and analytical solutions are presented for various test cases to show the effectiveness of observer technique.

\section*{Acknowledgements}
The research work done in this paper was supported by King Abdullah University of Science and Technology (KAUST), K.S.A.

\vspace{-1cm}
\begin{IEEEbiographynophoto}{Muhammad Usman Majeed}
is a PhD candidate in Computer, Electrical and Mathematical Sciences and Engineering (CEMSE) Division at King Abdullah University of Science and Technology (KAUST), Kingdom of Saudi Arabia. He joined KAUST in fall 2011 as an MS student and completed his masters in applied mathematics and computational sciences in December 2012. Before joining KAUST, Usman completed his Bachelors in Electrical Engineering in 2009 and worked as Lecturer in the Department of Electrical Engineering, University of Engineering and Technology (UET), Lahore Pakistan. Webpage: http://usmanmajeed.us/
\end{IEEEbiographynophoto}
%\newpage
\vspace{-1cm}
\begin{IEEEbiographynophoto}{Taous Meriem Laleg-Kirati} is an assistant professor in the division of Computer, Electrical and Mathematical Sciences and Engineering at KAUST.  She joined KAUST in December 2010. From 2009 to 2010, she was working as a research scientist at the French
Institute for research in Computer Sciences and Control Systems (INRIA) in Bordeaux. She earned her Ph.D in Applied Mathematics from INRIA Paris, in 2008. She holds a Master in control systems and signal
processing from University Paris 11 in France. Her research interests include, modeling, estimation, and control of physical systems, inverse problems, and signal and image analysis. She considers applications in engineering and biomedical fields. She is IEEE senior member. Webpage: http://emang.kaust.edu.sa
\end{IEEEbiographynophoto}


%\vfill


\end{document}